\documentclass[12pt]{amsart}
\usepackage{amsmath, amsfonts, amsbsy, amsthm, amscd, graphicx}
\usepackage{amssymb, latexsym, color}

\numberwithin{equation}{section}

\theoremstyle{plain}
\newtheorem{Theorem}{Theorem}[section]
\newtheorem{Proposition}[Theorem]{Proposition}
\newtheorem{Lemma}[Theorem]{Lemma}
\newtheorem{Corollary}[Theorem]{Corollary}

\theoremstyle{definition}
\newtheorem{Definition}[Theorem]{Definition}

\theoremstyle{remark}
\newtheorem{Remark}{{\bf Remark}}

\newtheorem{Conjecture}{Conjecture}

\newcommand{\R}{{\mathbb R}}

\newcommand{\Z}{{\mathbb Z}}
\newcommand{\N}{{\mathbb N}}

\newcommand{\bary}{\mathop{\rm bar}\nolimits}

\newcommand{\diam}{\mathop{\rm diam}\nolimits}

\newcommand{\girth}{\mathrm{girth}} 
\newcommand{\Dirac}{\mathrm{Dirac}} 
\newcommand{\sbt}{\,\begin{picture}(-.5,1)(-.5,-2)\circle*{1.5}\end{picture}\ }
\newcommand{\hilbert}{\mathcal{H}}

\newcommand{\defineeq}{\stackrel{\mathrm{def}}{=}}

\begin{document}

\title{Fixed-point property for %uniformly Lipschitz 
affine actions on a Hilbert space}

%\author{Hiroyasu Izeki$\mbox{}^1$} 
%\thanks{$\mbox{}^1$Faculty of Science and Technology, Keio University, 
%3-14-1 Hiyoshi, Kohoku-ku, Yokohama, Kanagawa 223-8522, Japan, 
%izeki@math.keio.ac.jp}

%\author{Takefumi Kondo$\mbox{}^2$}
%\thanks{$\mbox{}^2$Mathematics and Computer Science, Graduate School of Scinece and Engineering, 
%Kagoshima University, 
%1-21-35 Korimoto, Kagoshima-city, Kagoshima, 890-0065, Japan, takefumi@sci.kagoshima-u.ac.jp}

\author{Shin Nayatani}%$\mbox{}^3$}%*}
\thanks{%$\mbox{}^3$
Graduate School of Mathematics, Nagoya University,
Chikusa-ku, Nagoya 464-8602, Japan, nayatani@math.nagoya-u.ac.jp}
%\thanks{* Corresponding author}
%\address{Graduate School of Mathematics, Nagoya University, Chikusa-ku, Nagoya 464-8602, Japan}
%\email{nayatani@math.nagoya-u.ac.jp}
%\thanks{Partly supported by the Grant-in-Aid
%for Scientific Research, The Ministry of Education,
%Science, Sports and Culture, Japan.}
%\subjclass[2010]{Primary~20F65; Secondary~58E20, 20P05.}
%\keywords{finitely generated group, random group, $\cat$ space, fixed-point property, 
%energy of map, Wang invariant, expander, Euclidean building}

%\date{}

\maketitle

\begin{abstract}
Gromov \cite{Gromov2} showed that for fixed, arbitrarily large $C$, 
any uniformly $C$-Lipschitz affine action of a random group in his graph model 
on a Hilbert space has a fixed point. 
We announce a theorem stating that more general affine actions of the same 
random group on a Hilbert space have a fixed point. 
We discuss some aspects of the proof. 
\end{abstract} 

\section*{Introduction} 
In \cite{IzekiKondoNayatani2}, Izeki, Kondo and the present author proved that a random group 
in the Gromov graph model had fixed-point property, meaning that any isometric action had 
a fixed point, for a large class of CAT($0$) spaces, by using the method which concerns 
the $n$-step energy of maps. 
Naor and Silberman \cite{NaorSilberman} proved a similar result for a class of $p$-uniformly 
convex geodesic metric spaces. 
(Note that CAT($0$) spaces are $2$-uniformly convex.) 
In these studies it seemed that the condition that actions 
%of discrete groups on metric spaces 
are isometric was essential and without the condition the argument should break dwon. 
Gromov \cite{Gromov2}, however, had shown that 
%associated with a sequence of expanders with diverging girth, 
any uniformly $C$-Lipschitz affine action of the same random group on a Hilbert space 
has a fixed point, where $C$ may be arbitrarily large but should be specified in advance. 
The purpose of this article is to announce a fixed-point theorem for more general affine 
actions of the same random group, 
%(but the relevant sequence of expanders should satisfy stronger condition), 
allowing the Lipschitz constants of the affine maps to have mild growth 
with respect to a certain length function on the group \cite{IzekiKondoNayataniprep}. 
%This, in particular, verifies the above mentioned result of Gromov. 
It is worth while to mention the following: if the Lipschitz constants of the affine maps are 
uniformly bounded, then the action reduces to an isometric one on a Banach space 
by replacing the Hilbert norm by an equivalent one. 
On the other hand, our case treats really non-isometric actions which cannot reduce 
to isometric ones. 

A key of the proof is to verify the existence of a discrete harmonic map from the group 
into the Hilbert space which is equivariant with respect to the given action. 
In the case of isometric actions, the method of energy minimization coupled with scaling 
ultralimit argument was effective. In the general affine case, this method fails because 
a map minimizing local energy does not necessarily satisfy the condition of harmonicity.  
We therefore employ the method of discrete tension-contracting flow due to Gromov \cite{Gromov2}. 
Indeed, we refine Gromov's method and derive the existence of a harmonic map still 
by coupling it with scaling ultralimit argument. 

%当面の課題として, この手法をさらに改良して, リプシッツ定数が群上の語距離そのものに関して緩やかに
%増大する場合に, 同じランダム群が固定点をもつことを証明できるか見極めたい. 

This article is organized as follows. 
In \S 1, we define affine action, discuss the rigidity of isometric actions and state 
Shalom's theorem on the rigidity and existence of uniformly Lipschitz affine actions. 
In \S 2, after discussing Nowak's fixed-point theorem for uniformly Lipschitz affine actions 
of a random group in the Gromov density model, we state our main fixed-point theorem. 
In \S 3, we discuss discrete harmonic maps and state an existence theorem for such maps. 
We also discuss the failure of the method of energy minimization. 
In \S 4, we introduce Gromov's discrete tension-contracting flow and outline the proof of the existence 
of harmonic maps. 
In \S 5, we outline the proof of the main theorem. 
In Appendix, we prove the existence of maps minimizing local energy which are equivariant 
with respect to a given affine action. 

\section{%Uniformly Lipschitz 
Affine actions on a Hilbert space} 
Let $\hilbert$ be a Hilbert space, and denote the algebra of all bounded linear operators 
of $\hilbert$ by $\mathbb{B}(\mathcal{H})$. 
Let $\Gamma$ be a finitely generated infinite group, and let 
$\rho\colon \Gamma\curvearrowright \hilbert$ be an affine action. 
Thus, for $\gamma\in \Gamma$, $\rho(\gamma)\colon \hilbert\rightarrow \hilbert$ 
has the form
$$
\rho(\gamma)(\mathbf{v}) = A(\gamma)(\mathbf{v}) + b(\gamma),\quad \mathbf{v}\in \hilbert, 
$$
where $A(\gamma)\in \mathbb{B}(\mathcal{H})$ and $b(\gamma)\in \mathcal{H}$. 
Since $\gamma\mapsto \rho(\gamma)$ is a homomorphism, we have 
$$
A(\gamma\gamma') = A(\gamma)A(\gamma'),\quad b(\gamma\gamma') = b(\gamma) + A(\gamma)b(\gamma'),
\quad \gamma,\gamma'\in \Gamma. 
$$
%$$A(\gamma\gamma')(\mathbf{v}) + b(\gamma\gamma') 
%= A(\gamma) ( A(\gamma')(\mathbf{v}) + b(\gamma') ) + b(\gamma)$$

\begin{Definition}
An affine action $\rho\colon \Gamma\curvearrowright \hilbert$ is called 
{\em uniformly $C$-Lipschitz} if 
$\rho(\gamma)\colon \mathcal{H}\rightarrow \mathcal{H}$ is a $C$-Lipschitz map, 
or equivalently $\|A(\gamma)\|\leq C$, for all $\gamma\in \Gamma$. 
\end{Definition}

Note that $C\geq 1$ necessarily. 
%$$1=\|A(\gamma\gamma^{-1})\|\leq \|A(\gamma)\|\|A(\gamma^{-1})\|\leq C^2$$

Most basic example of a uniformly Lipschitz affine action is an isometric action. 
Recall that a $\sigma$-compact, locally compact topological group $G$ is said to have 
property $F\hilbert$ if any continuous isometric action 
$\rho\colon G \curvearrowright \mathcal{H}$ has a fixed point, that is, 
there exists $\mathbf{v}\in \hilbert$ such that $\rho(g)(\mathbf{v}) = \mathbf{v}$ for all $g\in G$. 
It is a celebrated result of Delorme \cite{Delorme} and Guichardet \cite{Guichardet} that 
property $F\hilbert$ is equivalent to Kazhdan's property (T). 
Kazhdan \cite{Kazhdan} defined this property for locally compact groups 
in terms of unitary representations, and proved that if $\Gamma$ is a lattice 
in a Lie group $G$, then $\Gamma$ has property (T) if and only if $G$ has property (T). 
As examples, simple real Lie groups of real rank at least two have property (T). 
For $n\geq 2$, $\mathrm{Sp}(n,1)$ is a simple Lie group of real rank one 
that has property (T). 
Thus these Lie groups and their lattices have property $F\hilbert$. 

%It is difficult to construct uniformly Lipschitz affine actions explicitly. 

In his unpublished work, Shalom proved the following theorem which exhibits 
that higher-rank groups have stronger rigidity than $\mathrm{Sp}(n,1)$ 
(cf.\! \cite{BFGM, Nowak2}). 

\begin{Theorem}[Shalom]%, unpublished]
\label{shalom_simple}
{\rm (i)}\,\, Any uniformly Lipschitz affine action of a simple real Lie group 
of real rank at least two (or its lattices) on $\hilbert$ has a fixed point.\\ 
{\rm (ii)}\,\, $\mathrm{Sp}(n,1)$ admits a uniformly Lipschitz affine action 
on $\mathcal{H}$ without fixed point. 
\end{Theorem}

Mimura \cite{Mimura} observes that the action in the statement (ii) %of the theorem 
is indeed metrically proper. 
Hence, any infinite discrete subgroup of $\mathrm{Sp}(n,1)$ also admits a uniformly 
Lipschitz affine action on $\mathcal{H}$ without fixed point. 
This exhibits many infinite hyperbolic groups which admit such affine actions. 

Shalom proposed the following (cf.\! \cite{Nowak2}) 

\begin{Conjecture}[Shalom]
Any non-elementary hyperbolic group admits a uniformly Lipschitz affine action 
on $\mathcal{H}$ without fixed point.
\end{Conjecture}

\section{Fixed-point property of random groups w.r.t.\! uniformly Lipschitz affine actions}

%BFGM, Mimuraの結果を確認

In this section, we review two fixed-point theorems regarding uniformly Lipschitz affine actions 
of certain ramdom groups on a Hilbert space. 
Recall that in the Gromov density model $\mathcal{G}(m,l,d)$ of random groups, 
generators $s_1^{\pm 1},\dots, s_m^{\pm 1}$ and a density $0<d<1$ are fixed, 
and choose $(2m-1)^{dl}$ words, each of them chosen uniformly and independently from the set 
of all reduced words of lenght $l$ in $s_1^{\pm 1},\dots, s_m^{\pm 1}$. 
The group $\Gamma$ generated by $s_1^{\pm 1},\dots, s_m^{\pm 1}$ and having those reduced words as relations 
is a constituent of the model $\mathcal{G}(m,l,d)$. 
Given a group property P (e.g.~Kazhdan's property (T)), we say that a {\em random group 
in the Gromov density model has property P} if the probability of $\Gamma$ having property P 
tends to one as $l\to \infty$. 

%Let $F_m$ be the free group on $m\geq 2$ generators $a_1,\dots,a_m$. 
%For any integer $l$ let $S_l\subset F_m$ be the set of reduced words of length $l$ in these generators. 
%
%Let $0\leq d\leq 1$. A random set of relators at density $d$, at length $l$ is a
%$(2m-1)^{dl}$-tuple of elements of $S_l$, randomly picked among all elements of
%$S_l$ (uniformly and independently).
%A random group at density $d$, at length $l$ is the group $G$ presented by
%$\langle a_1,\dots,a_m \mid R \rangle$ where $R$ is a random set of relators at density $d$, at length $l$. 
%We say that a property of $R$, or of $G$, occurs with overwhelming probability
%at density $d$ if its probability of occurrence tends to $1$ as $l\to \infty$, for fixed $d$.

\begin{Theorem}[Nowak \cite{Nowak1}]\label{nowak_theorem1} 
Fix $1\leq C<\sqrt{2}$. 
Let $\Gamma$ be a random group in the Gromov density model %$\mathcal{G}(m,d)$ 
with density $1/3 < d < 1/2$. 
Then any uniformly $C$-Lipschitz affine action 
$\rho\colon \Gamma\curvearrowright \mathcal{H}$ %with $1\leq C<\sqrt{2}$ 
has a fixed point. 
\end{Theorem}

Note that the random group $\Gamma$ of the theorem is non-elementary hyperbolic 
(hence infinite) \cite{Gromov1, Ollivier} and has property (T) 
\cite{Zuk2, KotowskiKotowski}. 

The proof of Theorem \ref{nowak_theorem1} is based on a fixed-point theorem 
for an isometric action of a deterministic group on a Banach space, 
which we shall review. 
Let $\Gamma$ be a finitely generated group equipped with a finite, symmetric 
generating set $S$ not containing the identity element. 
Modifying the construction as in \cite{Zuk2}, one constructs the link graph $\mathcal{L}(S)$; 
its vertices are the elements of $S$, generators $s$ and $t$ span an edge (written $s\sim t$) 
if $s^{-1}t$ is a generator, and the edges are suitably weighted. 
(For the account of the choice of weight, see \cite[p.~703]{Nowak1}, 
\cite[Proof of Lemma 3.1]{IzekiKondoNayatani1}.) 

Let $\mathcal{B}$ be a Banach space with norm $\|\cdot\|$ and denote by $\kappa_p(S, \mathcal{B})$ 
the optimal constant in the $p$-Poincar\'{e} inequality for maps $f\colon S\to \mathcal{B}$: 
\begin{equation*}\label{p-poincare}
\sum_{s\in S} \| f(s)-\overline{f} \|^p\, m(s)\leq \kappa^p \sum_{s\sim t} \| f(s)-f(t)\|^p\, m(s,t), 
\end{equation*}
where $m(s,t)$ is the weight of the edge $(s,t)$, $m(s) = \sum_{t\sim s} m(s,t)$, and 
$\overline{f} = \sum_{s\in S} m(s) f(s)/\sum_{s\in S} m(s)$, the mean value of $f$. 

\begin{Theorem}[Nowak \cite{Nowak1}]\label{nowak_theorem2}
Let $\mathcal{B}$ be a reflexive Banach space and let $\Gamma$ and $S$ be as above. 
If the link graph $\mathcal{L}(S)$ is connected and for some $1<p<\infty$ and its adjoint index 
$p^*$, satisfying $1/p+1/p^*=1$, the corresponding Poincar\'{e} constants satisfy 
$$
\max \{ 2^{-1/p} \kappa_p(S, \mathcal{B}), 2^{-1/p^*} \kappa_{p^*}(S, \mathcal{B}^*) \} < 1, 
$$
then any affine isometric action $\rho\colon \Gamma\curvearrowright \mathcal{B}$ 
has a fixed point. 
\end{Theorem} 

Let $\rho\colon \Gamma\curvearrowright \mathcal{H}$ be a uniformly $C$-Lipschitz 
affine action, where $\mathcal{H}$ is a Hilbert space, and introduce a new norm on $\mathcal{H}$ by 
$||| \mathbf{v} ||| = \sup_{\gamma\in \Gamma} \| A(\gamma) (\mathbf{v}) \|$ 
for $\mathbf{v}\in \mathcal{H}$. 
Then $\mathcal{B} = ( \mathcal{H}, ||| \cdot ||| )$ 
is a Banach space isomorphic to $\mathcal{H}$, thus reflexive, 
and $\rho\colon \Gamma\curvearrowright \mathcal{B}$ is an affine isometric action. 
Since the norms of $\mathcal{B}$ and $\mathcal{B}^*$ ($\cong \mathcal{H}$) satisfy 
$\|\cdot\|\leq |||\cdot |||\leq C \|\cdot\|$ and $C^{-1}\|\cdot\|\leq |||\cdot |||^*\leq \|\cdot\|$, respectively, 
it follows that  
%If $$\sum_{s\in S} ||| f(s)-\overline{f} |||^2\, m(s)\leq \kappa^2 \sum_{s\sim t} ||| f(s)-f(t) |||^2\, m(s,t), $$
%then $$\sum_{s\in S} \| f(s)-\overline{f} \|^2\, m(s)\leq C^2 \kappa^2 \sum_{s\sim t} \| f(s)-f(t) \|^2\, m(s,t). $$
%Therefore, 
$$
\kappa_2(S, \mathcal{B}), \kappa_2(S, \mathcal{B}^*)\leq C\, \kappa_2(S, \mathcal{H}) = C\, \kappa_2(S, \R). 
$$
Therefore, we obtain the following 

\begin{Corollary}\label{nowak_corollary}
Let $\Gamma$ and $S$ be as above, and suppose that the link graph $\mathcal{L}(S)$ is connected. 
Then any uniformly $C$-Lipschitz affine action $\rho\colon \Gamma\curvearrowright \mathcal{H}$ 
with $C\, \kappa_2(S, \R)<\sqrt{2}$ has a fixed point. 
\end{Corollary} 

Now let $\Gamma$ be a random group in the Gromov density model with density $1/3 < d < 1/2$. 
By the argument due to \.Zuk \cite{Zuk2}, Kotowski and Kotowski \cite{KotowskiKotowski}, 
%with high probability 
there is a random group $\Gamma'$ in a different model so that $\Gamma$ contains a quotient of $\Gamma'$ 
as a finite index subgroup and the link graph $\mathcal{L}(S')$ of $\Gamma'$ has $\kappa_2(S', \R)$ 
arbitrarily close to one. 
Therefore, we may apply Corollary \ref{nowak_corollary} to $\Gamma'$ and the conclusion 
of Theorem \ref{nowak_theorem1} holds for $\Gamma'$ and hence for $\Gamma$.  

Gromov \cite{Gromov2} claimed a result similar to Theorem \ref{nowak_theorem1} for a random group 
in the graph model which was also invented by him. 
To state Gromov's result, we first review the construction of this model. 
Let $F_m$ denote the free group on $m$ generators, and let $S$ be the collection of these $m$ elements 
and their inverses. 
Let $G = (V,E)$ be a finite connected graph, where $V$ and $E$ are the sets of vertices 
and undirected edges, respectively. 
We denote the set of directed edges by $\overrightarrow{E}$. 
A map $\alpha\colon \overrightarrow{E}\rightarrow S$ satisfying $\alpha((v,u)) = \alpha((u,v))^{-1}$ 
for all $(u,v)\in \overrightarrow{E}$ is called an {\em $S$-labelling} of $G$. 
Let $\mathcal{A}(m, G)$ denote the set of all $S$-labellings of $G$, consisting of $(2m)^{\# E}$ 
elements, and equip it with the uniform probability measure. 
For $\alpha\in \mathcal{A}(m, G)$
and a path $\overrightarrow{p} = (\overrightarrow{e}_1,\dots, \overrightarrow{e}_l)$ 
in $G$, where $\overrightarrow{e}_i \in \overrightarrow{E}$, define $\alpha(\overrightarrow{p}) 
= \alpha(\overrightarrow{e}_1)\cdot \dots \cdot \alpha(\overrightarrow{e}_l)\in F_m$. 
Then set 
\begin{align*}
& R_\alpha = \{ \alpha(\overrightarrow{c}) \mid \mbox{$\overrightarrow{c}$ are cycles in $G$} \},\\ 
& \Gamma_\alpha = F_m / \mbox{normal closure of $R_\alpha$}. 
\end{align*}
Let $\lambda_1(G, \R)$ denote the second eigenvalue of the discrete Laplacian of $G$, acting on 
real-valued functions on $V$. 
The {\em girth} of $G$, denoted by $\girth(G)$, is the minimal length of a cycle (i.e.~a closed path) in $G$. 
%and the {\em diameter} of $G$, denoted by $\diam(G)$, is the maximum distance between a pair of points in $G$. 

A sequence $\{ G_j \}_{j\in \N}$ of finite graphs is called a {\em sequence of (bounded-degree) 
expanders} if it satisfies 
\begin{enumerate}
\renewcommand{\theenumi}{\roman{enumi}}
\renewcommand{\labelenumi}{(\theenumi)}
\item $\# V_j\to \infty$ as $j\to \infty$,
\item $\exists d,\,\, \forall j,\,\, \forall u\in V_j,\,\, 2\leq {\rm deg} (u)\leq d$\quad (sparce), 
\item $\exists \lambda>0,\,\, \forall j,\,\, \lambda_1(G_j, \R) \geq \lambda$\quad (highly-connected). 
\end{enumerate} 
Such a $\{ G_j \}_{j\in \N}$ is said to have {\em diverging girth} if it further satisfies\\
\quad (iv)\,\, $\girth(G_j)\to \infty$ as $j\to \infty$.\\ 
Now suppose that a sequence of expanders $\{ G_j \}_{j\in \N}$ with diverging girth is given. 
%Let $\mathcal{A}_j$ be the set of all $S$-labelings of $G_j$, equipped with the uniform probability measure.  
%Note that $|\mathcal{A}_j| = \# S^{|E_j|}$. 
Then the collection of groups 
$\mathcal{G}(m, G_j) = \{ \Gamma_\alpha \mid \alpha\in \mathcal{A}(m, G_j) \}$ 
is the graph model of random groups. 
Given a group property P, we say that a {\em random group in the graph model has property P} 
if the probability of $\Gamma_\alpha$ having property P tends to one as $j\to \infty$. 
Gromov \cite{Gromov2} and Silberman \cite{Silberman1} proved that a random group in the graph model 
had fixed-point property for Hilbert spaces with respect to isometric actions, that is, it had 
property (T). 
This result was generalized to fixed-point property for CAT($0$) spaces \cite{IzekiKondoNayatani2} 
(see also \cite{Gromov2}) and for $p$-uniformly convex geodesic metric spaces \cite{NaorSilberman}. 
In both generalizations the degrees of singularity of the relevant geodesic metric spaces 
should be suitably bounded. 

We now state

\begin{Theorem}[Gromov \cite{Gromov2}]\label{main_fixed_point_theorem1} 
Fix $C>0$. 
Let $\Gamma$ be a random group in the graph model associated with a sequence of expanders 
with diverging girth. 
Then any uniformly $C$-Lipschitz affine action $\rho\colon \Gamma\curvearrowright \mathcal{H}$ 
has a fixed point. 
\end{Theorem}

We can relax the condition that the Lipschitz constants of the relevant affine maps should be 
uniformly bounded. 
To state our result precisely, we introduce the following 

\begin{Definition}
Let $\Gamma$ be a finitely generated group equipped with a finite, symmetric generating set $S$, 
and let $l\colon \Gamma\rightarrow \Z_{\geq 0}$ denote the word-length function with respect to $S$. 
For each conjugacy class $c$ of $\Gamma$, we define
$$
l_{\mathrm{conj}}(c) = \inf_{\gamma\in c} l(\gamma)
$$
and call $l_{\mathrm{conj}}\colon \{ \mbox{conjugacy classes of $\Gamma$} \}\rightarrow \Z_{\geq 0}$
the {\em conjugacy-length function} of $\Gamma$ \cite{CoornaertKnieper}. 
\end{Definition}

We now state

\begin{Theorem}\label{main_fixed_point_theorem2} 
Fix $C>0$ and $0\leq \sigma<1/10$. 
Let $\Gamma$ be a random group in the graph model associated with a sequence of expanders 
with diverging girth and diameter growing at most linearly in girth. 
Then any affine action $\rho\colon \Gamma\curvearrowright \mathcal{H}$ 
satisfying 
\begin{equation}\label{mildgrowth}
\forall \gamma\in \Gamma,\,\, \|A(\gamma)\|\leq C\,l_{\mathrm{conj}}([\gamma])^\sigma, 
\end{equation}
where $[\gamma]$ denotes the conjugacy class containing $\gamma$, has a fixed point. 
\end{Theorem}

%conjugacy-lengthをlengthに置き換えられるかは今後の課題であることをコメントした方がよい。

\begin{Remark}
In order for a random group in the graph model to be non-elementary hyperbolic 
(hence infinite), the relevant sequence of expanders should satisfy some further 
conditions (cf.~ \cite{Gromov2, Ghys, ArzhantsevaDelzant}). 
One of these conditions implies the diameter growth condition in Theorem 
\ref{main_fixed_point_theorem2}, which is therefore essentially superficial. 
\end{Remark}

%\begin{Remark}
%It is expected to follow from a result of V.\! Lafforgue (2008) that 
%any non-elementary hyperbolic group $\Gamma$ admits an affine action without fixed point 
%on a Hilbert space such that 
%$$
%\| A(\gamma) \|\leq \mathrm{polynomial}(\mathrm{word\, length} (\gamma)).%,\,\, \gamma\in \Gamma. 
%$$
%\end{Remark}

\section{Discrete harmonic maps} 

Let $\Gamma$ be a finitely generated group and fix a finite, symmetric generating set $S$. 
Let $\mu$ be the standard random walk of $\Gamma$ associated with $S$, that is, 
$$
\mu(x\to x')\defineeq \left\{\begin{array}{cl} 1/\# S & \mbox{if $\exists s\in S$, $x'=xs$},\\
0 & \mbox{otherwise}.\end{array}\right.
$$
The {\em barycenter map} of $\hilbert$, $\bary\colon \{ \mbox{finite-support probability measures 
on $\mathcal{H}$} \} \rightarrow \hilbert$, is given by 
\begin{equation}\label{affine_barycenter}
\bary\left( \sum_{i=1}^m t_i \Dirac(\mathbf{v}_i) \right) = \sum_{i=1}^m t_i \mathbf{v}_i. 
\end{equation}
Let $\rho\colon \Gamma\curvearrowright \mathcal{H}$ be an 
%uniformly $C$-Lipschitz 
affine action. 

\begin{Definition}
A $\rho$-equivariant map $f\colon \Gamma\rightarrow \mathcal{H}$ is called {\em harmonic} if 
it satisfies 
\begin{equation}\label{harmonic}
\bary(f_*\mu(x\to \sbt)) = f(x)
\end{equation}
for all $x\in \Gamma$. Notice that 
$$
\bary(f_*\mu(x\to \sbt)) = \frac{1}{\# S}\sum_{s\in S} f(xs). 
$$
\end{Definition}

\begin{Remark}\label{role-of-affinity}
Since the action $\rho$ is affine, we may conclude that a $\rho$-equivariant $f$ is harmonic 
if it satisfies \eqref{harmonic} for {\em some} $x\in \Gamma$. 
To see this, suppose that \eqref{harmonic} holds for $x$, and write any $x'\in \Gamma$ as $x'=\gamma x$. 
Then
%since
%$$\bary(f_*\mu(x\to \sbt)) = \frac{1}{\# S}\sum_{s\in S} f(xs),$$
\begin{eqnarray*}
\bary(f_*\mu(x'\to \sbt)) &=& \frac{1}{\# S}\sum_{s\in S} f(x's) 
= \frac{1}{\# S}\sum_{s\in S} \rho(\gamma)(f(xs))\\ 
&=& \rho(\gamma) \left( \frac{1}{\# S}\sum_{s\in S} f(xs) \right) 
= \rho(\gamma) f(x)\\ 
&=& f(x'),  
\end{eqnarray*}
%\begin{eqnarray*}
%\bary(f_*\mu(x'\to \sbt)) &=& \bary((\rho(\gamma)\circ f)_* \mu(x\to \sbt))\\ 
%&=& \frac{1}{\# S}\sum_{s\in S} \rho(\gamma)(f(xs))\\ 
%&=& \rho(\gamma) \left( \frac{1}{\# S}\sum_{s\in S} f(xs) \right)\\ 
%&=& \rho(\gamma) ( \bary(f_*\mu(x\to \sbt))\\ 
%&=& \rho(\gamma) f(x) = f(x'),  
%\end{eqnarray*}
and \eqref{harmonic} holds for $x'$, too. 
\end{Remark}

The action $\rho$ has a fixed point if and only if a $\rho$-equivariant constant map, 
which are trivially harmonic, exists. 
In contrast, we have the following existence result for nonconstant harmonic maps 
when $\rho$ has no fixed point.  

\begin{Theorem}\label{existence}
Let $\Gamma$ be a finitely generated group equipped with a finite, symmetric generating set $S$, 
and let $\rho\colon \Gamma\curvearrowright \mathcal{H}$ be an affine action, 
where $\mathcal{H}$ is a Hilbert space, satisfying \eqref{mildgrowth} for some $C>0$ and $\sigma\geq 0$. 
Suppose that $\rho(\Gamma)$ has no fixed point. 
Then there exist a (possibly new) affine action $\rho'\colon \Gamma\curvearrowright \hilbert'$, 
where $\hilbert'$ is a (possibly new) Hilbert space, satisfying \eqref{mildgrowth} 
for the same $C$, $\sigma$ as above 
and a nonconstant harmonic $\rho'$-equivariant map $f\colon \Gamma \rightarrow \mathcal{H}'$. 
\end{Theorem}

Before discussing the actual proof, we observe that the standard approach via energy minimization 
coupled with scaling ultralimit argument would fail.

\begin{Definition}\label{def-of-local-energy}
For a $\rho$-equivariant map $f\colon \Gamma\rightarrow \mathcal{H}$ and $x\in \Gamma$, 
define the {\em local energy} $E(f)(x)$ of $f$ at $x$ by
\begin{equation}\label{energy}
E(f)(x) \defineeq \frac{1}{2} \sum_{x'\in \Gamma} \|f(x)-f(x')\|^2\, \mu(x\to x'). 
\end{equation}
\end{Definition}

As will be verified in the appendix, under the assumption that $\rho(\Gamma)$ has no fixed point, 
one can always find a nonconstant $\rho$-equivariant map $f\colon \Gamma\rightarrow \hilbert$ 
minimizing the local energy at $x$, 
though the Hilbert space $\hilbert$ and the affine action 
$\rho\colon \Gamma\curvearrowright \mathcal{H}$ should possibly be renewed. 
We now focus on the question whether the map $f$ satisfies \eqref{harmonic} for $x$. 
%For simplicity, we set $f:=f_\infty$ and omit the prime symbol ($\mbox{}^\prime$) from the notation. 
For $\mathbf{v}\in \hilbert$ and $t\in \R$, let $f_t\colon \Gamma\rightarrow \hilbert$ 
be the $\rho$-equivariant map such that $f_t(e) = f(e) + t \mathbf{v}$. 
Then we have, taking $x=e$ for simplicity, 
\begin{eqnarray*}
E(f_t)(e) &=& \frac{1}{2 \# S} \sum_{s\in S} \|f_t(e)-f_t(s)\|^2\\ 
&=& \frac{1}{2 \# S} \sum_{s\in S} \|f(e)+t\mathbf{v}-\rho(s)(f(e)+t\mathbf{v})\|^2\\ 
%&=& \frac{1}{2 \# S} \sum_{s\in S} \|f(e)+t\mathbf{v}-\rho(s)(f(e)) - t A(s)(\mathbf{v})\|^2\\ 
&=& \frac{1}{2 \# S} \sum_{s\in S} \|(f(e)-f(s)) + t(\mathbf{v}-A(s)(\mathbf{v}))\|^2\\ 
&=& \frac{1}{2 \# S} \sum_{s\in S} \left( \|(f(e)-f(s))\|^2 
+ 2t \left\langle f(e)-f(s), \mathbf{v}-A(s)(\mathbf{v}) \right\rangle + O(t^2) \right), 
\end{eqnarray*}
and therefore 
$$
0 = \frac{d}{dt} E(f_t)(e)|_{t=0} = \frac{1}{\# S} \sum_{s\in S}  \left\langle f(e)-f(s), 
\mathbf{v}-A(s)(\mathbf{v}) \right\rangle. 
$$
If the action $\rho$ is isometric, which means that $A(s)$ is orthogonal, then 
\begin{eqnarray*}
\mbox{R.H.S.} &=& \frac{1}{\# S} \sum_{s\in S}  \left\langle f(e)-f(s), \mathbf{v} \right\rangle 
- \frac{1}{\# S} \sum_{s\in S}  \left\langle A(s^{-1})(f(e)-f(s)), \mathbf{v} \right\rangle\\ 
&=& \frac{1}{\# S} \sum_{s\in S}  \left\langle f(e)-f(s), \mathbf{v} \right\rangle 
- \frac{1}{\# S} \sum_{s\in S}  \left\langle f(s^{-1})-f(e), \mathbf{v} \right\rangle\\ 
&=& \frac{2}{\# S} \sum_{s\in S}  \left\langle f(e)-f(s), \mathbf{v} \right\rangle. 
\end{eqnarray*}
Since this vanishes for all $\mathbf{v}\in \hilbert$, we conclude \eqref{harmonic}. 
However, if $\rho$ is not isometric, the above computation fails and we would not be able 
to conclude \eqref{harmonic}, that is, that $f$ is harmonic. 

Instead, we use Gromov's discrete tension-contracting flow developed in \cite[\S 3.6]{Gromov2} 
and produce a harmonic $f$. 
Postponing the details to \cite{IzekiKondoNayataniprep}, we shall outline the argument 
for the proof of Theorem \ref{existence}. In the remainder of this section, 
let $\Gamma$ be a finitely generated group equipped with a finite, symmetric generating set $S$, 
and let $\rho\colon \Gamma\curvearrowright \mathcal{H}$ be an affine action, 
where $\mathcal{H}$ is a Hilbert space. 

For a $\rho$-equivariant map $f\colon \Gamma\rightarrow \mathcal{H}$, %and $0<\varepsilon< 1$, 
define new maps 
$Hf\colon \Gamma\rightarrow \mathcal{H}$ and 
$\Delta f\colon \Gamma\rightarrow \mathcal{H}$ by 
\begin{eqnarray*}
Hf (x) &\defineeq& \frac{1}{2} \left( \sum_{x'\in \Gamma} f(x')\, \mu(x \to x') + f(x) \right)\\ 
&=& \frac{1}{2} \left( \frac{1}{\# S}\sum_{s\in S} f(xs) + f(x) \right), 
\end{eqnarray*}
\begin{eqnarray*}
\Delta f (x) &\defineeq& (1-H)f (x) 
%= \varepsilon \left( f(x) - \int_\Gamma f(x')\, d\mu(x \to x') \right)
\\ 
&=& \frac{1}{2} \sum_{x'\in \Gamma} ( f(x) - f(x') )\, \mu(x \to x')\\ 
&=& \frac{1}{2 \# S} \sum_{s\in S} ( f(x) - f(xs) ). 
\end{eqnarray*}
%respectively. 
The maps $Hf$ and $\Delta f$ are $\rho$-equivariant and $A$-equivariant, respectively. 
We call $H$ (resp. $\Delta$) the {\em averaging operator} (resp. {\em Laplacian}). 
Note that $f$ is harmonic if and only if $\Delta f=0$, or $Hf=f$. 

\begin{Proposition}[cf.~Gromov \cite{Gromov2}]\label{prop:monotone_decreasing}
We have 
\begin{equation*}
\| \Delta H f (x) \|\leq 
\max_{x' \in x(S\cup \{e\})} \| \Delta f(x') \| 
\end{equation*}
for all $x\in \Gamma$, and if the equality sign holds for some $x$ 
%%%%%%%%%%%%%%%%%%%%%%%%%%%%%%%%
%for some x でよいか？証明を確認
%%%%%%%%%%%%%%%%%%%%%%%%%%%%%%%%
then $\Delta f (x)$ is a constant vector independent of $x\in \Gamma$. 
\end{Proposition}

Motivated by this proposition, we introduce the following 

\begin{Definition}[cf.~Gromov \cite{Gromov2}]
Let $f\colon \Gamma\rightarrow \mathcal{H}$ be a $\rho$-equivariant map, and define 
$f_0 := f$ and $f_{i+1} := H f_i$ inductively. 
We say that $f$ is {\em (harmonically) stable} if 
\begin{equation*}
0 < \exists \lambda < 1, \ 
\exists i_0 \in \N, \ 
%\text{ s.t. }
\forall i\geq i_0, \ 
%\in \N \text{ with } i\geq i_0, 
\forall x \in \Gamma, 
\ \| \Delta f_{i+1}(x) \|
    \leq \lambda \max_{x' \in x(S\cup \{e\})}
\| \Delta f_{i}(x') \|. 
\end{equation*}
\end{Definition}

It should be mentioned that the above definition of harmonic stability is slightly modified 
from Gromov's original one and it is more suitable for our purpose. 

\begin{Remark}\label{harmonic-is-stable} 
Suppose $f_{i_0}$ is harmonic, that is, $\Delta f_{i_0} = f_{i_0} - f_{i_0+1} = 0$ 
for some $i_0$. 
Then $f_{i} = f_{i_0}$ and thus $\Delta f_{i}=0$ for all $i\geq i_0$. 
Therefore, $f$ is stable. 
\end{Remark}

\begin{Proposition}\label{stable_to_harmonic} 
Suppose that $\rho$ satisfies \eqref{mildgrowth} for some $C>0$ and $\sigma\geq 0$. 

\noindent
{\rm (i)}\, 
If a $\rho$-equivariant map $f\colon \Gamma\rightarrow \mathcal{H}$ is stable, 
then 
%for any $x \in \Gamma$, $\{f_i (x)\}_{i\in \N}$ is a Cauchy sequence. 
$\{f_i\}_{i\in \N}$ converges pointwise to a map 
$f_{\infty} \colon \Gamma \rightarrow \mathcal{H}$, and $f_\infty$ is harmonic. 

\noindent
{\rm (ii)}\, 
If a $\rho$-equivariant map $f\colon \Gamma\rightarrow \mathcal{H}$ is not stable, 
then there exist a Hilbert space $\hilbert'$ 
%isometrically isomorphic to $\hilbert$ 
%(or its codimension one subspace if $\hilbert$ is finite-dimensional) 
and a nonconstant harmonic map $f'\colon \Gamma\rightarrow \hilbert'$, 
equivariant with respect to an affine action $\rho'\colon \Gamma\curvearrowright \hilbert'$ 
satisfying \eqref{mildgrowth} for the same $C$, $\sigma$ as above. 
\end{Proposition}

This proposition implies Theorem \ref{existence}. 
The proofs of the two propositions above will be given in \cite{IzekiKondoNayataniprep}. 

\section{%Outline of the 
Proof of Theorem \ref{main_fixed_point_theorem2}}

In this section, we prove Theorem \ref{main_fixed_point_theorem2}. 
Let $\Gamma$ be a finitely generated group equipped with a finite, symmetric generating 
set $S$. 
Let $\rho\colon \Gamma\curvearrowright \mathcal{H}$ be an affine action, 
where $\mathcal{H}$ is a Hilbert space, and suppose that $\rho$ satisfies 
\begin{equation}\label{eq:growth-new'}
\forall \gamma\in \Gamma,\,\, \| A(\gamma)\|\leq C\,l(\gamma)^{\sigma}
\end{equation} 
for some $C>0$ and $\sigma\geq 0$. 
(Note that this condition is weaker than \eqref{mildgrowth}.) 
For a $\rho$-equivariant map $f\colon \Gamma\rightarrow \mathcal{H}$ and $x\in \Gamma$, 
define the {\em local $n$-step energy} of $f$ at $x$ by
$$
E^{(n)}(f)(x) \defineeq \frac{1}{2} \sum_{x'\in \Gamma} \|f(x)-f(x')\|^2\, \mu^n(x\to x'),  
$$
where $\mu^n$ is the $n$-th convolution of $\mu$. 

\begin{Lemma}\label{growth} 
Suppose that $\sigma< 1/2$. 
Let $f\colon \Gamma\rightarrow \mathcal{H}$ be a harmonic $\rho$-equivariant map. 
Then we have 
%its local $n$-step energy satisfies 
$$
E^{(n)}(f)(x) \gtrsim_{C,\sigma,x} n^{1-2\sigma} E(f)(x) 
$$
for all $x\in \Gamma$.
\end{Lemma} 

The proof of this lemma will be given in \cite{IzekiKondoNayataniprep}. 

So far, the group $\Gamma$ has been any finitely generated group. 
The following lemma, essentially due to Gromov and Silberman \cite{Gromov2, Silberman1}, 
concerns a random $\Gamma$. 

\begin{Lemma}\label{nongrowth} 
Let $\Gamma$ be a random group in the graph model associated with a sequence of expanders 
with diverging girth and diameter growing at most linearly in girth, 
%Let $\Gamma$ be a random group as in Theorem \ref{main_fixed_point_theorem2}, 
and let $\rho\colon \Gamma\curvearrowright \hilbert$ be an affine action as above. 
Then for any $\rho$-equivariant map $f\colon \Gamma\rightarrow \mathcal{H}$ 
and any $x\in \Gamma$, we have 
\begin{equation}\label{nstepenergy-upperbound}
E^{(n)}(f)(x)\lesssim_{C,\sigma,x,\lambda} n^{8\sigma} E(f)(x). 
\end{equation}
Here, $n$ is a positive integer depending on $f$ and $x$, and we may assume that $n$ 
is arbitrarily large, and $\lambda$ is the positive constant as in the definition 
of a sequence of expanders. 
%, which depends on $f$ and $x$ but may be chosen arbitrarily large. 
\end{Lemma}

\begin{proof} 
The proof of Proposition 2.14 in \cite{Silberman1}, which treats the case that the action 
is isometric, mostly works for the non-isometric case. 
For the readers' convenience we outline Silberman's argument, and explain 
how the term $n^{8\sigma}$ comes in. 

The lemma is a consequence of the following statement. 
With probability tending to one as $j\to \infty$, 
the group $\Gamma$ corresponding to $\alpha\in \mathcal{A}(m, G_j)$ has the following 
property: for any affine action $\rho\colon \Gamma\curvearrowright \hilbert$ 
satisfying \eqref{eq:growth-new'}, any $n<\frac{\girth(G_j)}{2}$, any $\rho$-equivariant map 
$f\colon \Gamma \rightarrow \mathcal{H}$ and any $x\in \Gamma$, there exists 
$\sqrt{n} < l\leq n$ such that 
\begin{equation}\label{nstepenergy-upperbound2}
E^{(l)}(f)(x)\lesssim_{C,\sigma, x,\lambda} \diam(G_j)^{4\sigma}\, E(f)(x). 
%E^{(l)}(f)(x)\leq \frac{8\, C^{12}\, \diam(G_j)^{4\sigma}\, l(x)^{8\sigma}}{C'\lambda} E(f)(x), 
\end{equation}
%where $C'$ is a absolute positive constant. 
%there exists $n\gtrsim \sqrt{\girth(G_j)}$ such that \eqref{nstepenergy-upperbound} holds. 
Indeed, choosing $n\simeq \frac{\girth(G_j)}{2}$, we have 
$\diam(G_j)\lesssim n\leq l^2$ by the assumption on $\diam(G_j)$, and therefore 
$$
E^{(l)}(f)(x)\lesssim_{C,\sigma,x,\lambda} l^{8\sigma} E(f)(x). 
$$
Note that $l\gtrsim \sqrt{\girth(G_j)}$; thus $l$ diverges as $j\to \infty$. 

For the time being, fix a member $G_j$ of the expander sequence defining the graph model, 
and denote it by $G = (V,E)$. 
Let $\mu_G$ and $\nu_G$ denote the standard random walk on $G$ 
and the standard probability measure on $V$ given by 
$$
\mu_G(u,v)=\left\{\begin{array}{cl} \frac{1}{\deg(u)} & \mbox{if $(u,v) \in \overrightarrow{E}$,}\\
0 & \mbox{otherwise,} \end{array} \right.\quad \mbox{and}\quad 
\nu_G(u)=\frac{\deg(u)}{2\# E}, 
$$
respectively. 
For a map $\varphi\colon V \rightarrow \hilbert$ and $n\in \N$, 
the {\em $n$-step energy} of $\varphi$ is defined by 
\begin{equation*}
E_{\mu_G^n}(\varphi) = \frac{1}{2} \sum_{u \in V}\nu_G(u) \sum_{v \in V} 
\|\varphi(u)-\varphi(v)\|^2\, \mu_G^n(u\to v). 
\end{equation*} 
Recall \cite[Lemma 2.11]{Silberman1} that we have
\begin{equation}\label{energy-inequality}
E_{\mu_G^n} (\varphi)\leq \frac{2}{\lambda_1(G,\R)} E_{\mu_G} (\varphi) 
\end{equation}
for all maps $\varphi\colon V\rightarrow \hilbert$ and all $n\in \N$. 

Let $\alpha\colon \overrightarrow{E}\rightarrow S$ be an $S$-labelling of $G$, 
and $\Gamma$ the corresponding group. 
Let $\rho\colon \Gamma\curvearrowright \hilbert$ be an affine action, 
and $\widetilde{\rho}\colon F_m\curvearrowright \hilbert$ its lift. 
The strategy in proving \eqref{nstepenergy-upperbound2} is to transplant 
\eqref{energy-inequality} onto $\Gamma$. 
In fact, we may work on $F_m$ instead of $\Gamma$, and so we shall transplant 
\eqref{energy-inequality} onto $F_m$. 
In order to do this, we `push-forward', using $\alpha$, the random walks $\mu_G$ 
and $\mu_G^n$ on $G$ to those on $F_m$ as follows. 

If $u\in V$ and $x\in F_m$ are fixed, $\alpha$ induces a corresponding graph morphism 
$\beta_{u\to x}$ from $G$ to $X=\mathrm{Cay}(F_m, S)$, the Cayley graph of $F_m$ 
with respect to $S$, as follows: 
For $v\in V$, choose a path 
$\overrightarrow{p} = (\overrightarrow{e}_1,\dots, \overrightarrow{e}_l)$ 
from $u$ to $v$ in $G$, and set 
$$
\beta_{u\to x} (v) \defineeq x\, \alpha(\overrightarrow{p}) 
= x\, \alpha(\overrightarrow{e}_1)\cdot \dots \cdot \alpha(\overrightarrow{e}_l). 
$$
To be precise, $\beta_{u\to x}$ is well-defined only on %$\stackrel{\circ}{B}\!(u, g/2)$, 
the set of vertices whose graph distance from $u$ is less than $g/2$, where $g=\girth(G)$. 
%rooted treeと書きたい。
We now define, for $n<g/2$, a random walk $\mu_{G,\alpha}^n$ on $X$ 
%which can be regarded as the push-forward of the $n$-th convolution $\mu_G^n$ 
%of the random walk $\mu_G$ in terms of $\alpha$, 
by 
$$
\mu_{G,\alpha}^n(x\to \cdot) \defineeq \sum_{u\in V} \nu_G(u)\, 
(\beta_{u\to x})_* \mu_G^n(u\to \cdot). 
$$
Note that the average over $V$ is taken in order to produce a random walk 
independent of the individual vertices of $G$. 

We can now transplant \eqref{energy-inequality} onto $F_m$, and it is here that 
something different occurs when the action $\rho$ is non-isometric. 
Suppose $n<g/2$ and let $f\colon F_m\rightarrow \hilbert$ be a 
$\widetilde{\rho}$-equivariant map.\footnote{
Equivalently, $f\colon F_m\rightarrow \hilbert$ is the lift of a $\rho$-equivariant map 
$\Gamma\rightarrow \hilbert$. 
In particular, the map $f\circ \beta_{u\to x}$ is well-defined on the whole vertex set $V$.
}
When $\rho$ is isometric, 
\begin{equation}\label{energy-relation-isometric}
E_{\mu_{G,\alpha}^n}(f)(x) = E_{\mu_G^n}(f\circ \beta_{u_0\to x})(x) 
\end{equation}
holds for a fixed $u_0\in V$. 
Indeed, 
\begin{eqnarray*} 
E_{\mu_{G,\alpha}^n}(f)(x) &=& \frac{1}{2} \sum_{u\in V} \nu_G(u)\, 
\sum_{x'\in F_m} \| f(x) - f(x')\|^2\, [(\beta_{u\to x})_* \mu_G^n(u\to \cdot)] (x')\\ 
&=& \frac{1}{2} \sum_{u\in V} \nu_G(u)\, 
\sum_{v\in V} \| f\circ \beta_{u\to x} (u) - f\circ \beta_{u\to x} (v) \|^2\, \mu_G^n(u\to v). 
\end{eqnarray*} 
If $\rho$ is isometric, then we can replace $\beta_{u\to x}$ by $\beta_{u_0\to x}$ 
in the last expression and get the right-hand side of \eqref{energy-relation-isometric}. 
Now consider the general case that $\rho$ is not necessarily isometric. 
Let $\overrightarrow{p}$ and $\overrightarrow{r}$ be a path from $u_0$ to $v$ and 
a shortest path from $u$ to $u_0$, respectively, and let $\overrightarrow{q}$ denote the path 
from $u$ to $v$ traveling along $\overrightarrow{r}$ and $\overrightarrow{p}$ in this order. 
Then 
\begin{eqnarray*} 
f\circ \beta_{u\to x} (v) &=& f(x\alpha(\overrightarrow{q}))\\ 
&=& \widetilde{\rho}(x\alpha(\overrightarrow{r})x^{-1}) f(x\alpha(\overrightarrow{p}))\\ 
&=& \widetilde{\rho}(x\alpha(\overrightarrow{r})x^{-1}) f\circ \beta_{u_0\to x} (v), 
\end{eqnarray*} 
and therefore 
$$ 
\| f\circ \beta_{u\to x} (u) - f\circ \beta_{u\to x} (v) \| 
\leq \| \widetilde{A}(x\alpha(\overrightarrow{r})x^{-1}) \| 
\| f\circ \beta_{u_0\to x} (u) - f\circ \beta_{u_0\to x} (v) \|, 
$$
where $\widetilde{A}$ is the linear part of $\widetilde{\rho}$. 
Since
\begin{eqnarray*}
\| \widetilde{A}(x\alpha(\overrightarrow{r})x^{-1}) \|
&\leq& \| \widetilde{A}(x) \| \| \widetilde{A}(\alpha(\overrightarrow{r})) \| 
\| \widetilde{A}(x^{-1}) \|\\ 
&\leq& C^3\, l(x)^\sigma\, l(\alpha(\overrightarrow{r}))^\sigma\, l(x^{-1})^\sigma\\ 
&\leq& C^3\, D^\sigma\, l(x)^{2\sigma}, 
\end{eqnarray*}
where $D=\diam(G)$, we obtain 
$$
E_{\mu_{G,\alpha}^n}(f)(x)\leq C^6\, D^{2\sigma}\, l(x)^{4\sigma}
E_{\mu_G^n}(f\circ \beta_{u_0\to x})(x), 
$$
and likewise, 
$$
E_{\mu_G}(f\circ \beta_{u_0\to x})(x)\leq C^6\, D^{2\sigma}\, l(x)^{4\sigma}
E_{\mu_{G,\alpha}}(f)(x). 
$$
Together with \eqref{energy-inequality}, these imply 
\begin{equation}\label{tp-energy-inequality}
E_{\mu_{G,\alpha}^n}(f)(x)\leq \frac{2\, C^{12}\, D^{4\sigma}\, l(x)^{8\sigma}}{
\lambda_1(G,\R)} E_{\mu_{G,\alpha}}(f)(x). 
\end{equation}

In order to conclude \eqref{nstepenergy-upperbound2} 
(provisionally on $F_m$ instead of $\Gamma$), 
we must show that with high probability the random walks $\mu_{G,\alpha}$ and $\mu_{G,\alpha}^n$ 
in \eqref{tp-energy-inequality} can be replaced by $\mu_X$ and $\mu_X^l$, $\sqrt{n}< l\leq n$, 
respectively, where $\mu_X$ is the standard random walk of $X$. 
This will be done by verifying that with high probability the random variables
$\alpha\mapsto \mu_{G,\alpha}$ and $\alpha\mapsto \mu_{G,\alpha}^n$ concentrate on 
their expectations and that these expectations are computed in terms of 
$\mu_X$ and its convolutions. 

We begin with the second issue. 
For $n<g/2$, the expectation $\overline{\mu}_{G,X}^n$ of the random variable 
$\alpha\mapsto \mu_{G,\alpha}^n$ can be computed and expressed as a convex combination 
of $\mu_X^l$, $0\leq l\leq n$: 
%引用をつける。
\begin{equation}\label{expectation-expression}
\overline{\mu}_{G,X}^n = \sum_{l=0}^ n w^{(n)}_{\,\,\,l}\, \mu_X^l, 
\end{equation}
where the weights $w^{(n)}_{\,\,\,l}$ satisfy 
\begin{equation}\label{expectation-concentration}
\sum_{\sqrt{n}<l\leq n} w^{(n)}_{\,\,\,l} \geq C'
\end{equation}
for a certain absolute constant $C'>0$. 

For the first issue, let $j$ get large and observe that the random variables $\mu_{G_j,\sbt}$ 
and $\mu_{G_j,\sbt}^n$, where $n<g_j/2$, 
%with $n$ suitably chosen 
concentrate on their expectations 
$\overline{\mu}_{G_j,X}$ and $\overline{\mu}_{G_j,X}^n$, respectively. 
Indeed, one can verify that the map $\alpha\mapsto \mu_{G_j,\alpha}^n$ is Lipschitz 
with respect to the Hamming distance on $\mathcal{A}(m, G_j)$ with the Lipschitz 
constant depending only on the fixed parameters $d$, $m$. 
Using this fact, one deduces that with probability tending to one as $j\to \infty$,  
%\begin{equation}\label{concentration_1}
$$
\mu_{G_j,\alpha}(x\to x')\leq 2\, \overline{\mu}_{G_j,X}(x\to x')\quad 
\mbox{and}\quad 
\mu_{G_j,\alpha}^n(x\to x')\geq \frac{1}{2}\, \overline{\mu}_{G_j,X}^n(x\to x') 
$$
hold for all $x,x'\in X$. 
%\end{equation}
%where $n=\min \{ C_{d,k}\log \# V_j, g_j/2 \}$, $g_j=g_j$. 
%Silbermanの元の証明をチェックしておきたい。

Now for any $\widetilde{\rho}$-equivariant map $f\colon F_m\rightarrow \hilbert$, we obtain 
\begin{equation*} 
E_{\mu_{G_j,\alpha}}(f)(x) \leq 2\, E_{\overline{\mu}_{G_j,X}}(f)(x) = 2\, E_{\mu_X}(f)(x)
\end{equation*}
and 
\begin{eqnarray*} 
E_{\mu_{G_j,\alpha}^n}(f)(x) &\geq& \frac{1}{2}\, E_{\overline{\mu}_{G_j,X}^n}(f)(x) 
\geq \frac{1}{2} \sum_{\sqrt{n}<l\leq n} w^{(n)}_{\,\,\,l}\, E_{\mu_X^l}(f)(x)\\ 
%&\geq& \frac{1}{2} \left( \sum_{\sqrt{n}<l\leq n} w^{(n)}_{\,\,\,l} \right) 
%\min_{\sqrt{n}<l\leq n} E_{\mu_X^l}(f)\\ 
&\geq& \frac{C'}{2} \min_{\sqrt{n}<l\leq n} E_{\mu_X^l}(f)(x). 
\end{eqnarray*}
Together with \eqref{tp-energy-inequality}, these imply that there exists $\sqrt{n}<l\leq n$ 
(which depends on $f$ and $x$) such that 
$$
E_{\mu_X^l}(f)(x)\leq \frac{8\, C^{12}\, D_j^{4\sigma}\, l(x)^{8\sigma}}{C' \lambda} 
E_{\mu_X}(f)(x). 
$$

Now let $f\colon \Gamma\rightarrow \hilbert$ be a $\rho$-equivariant map and set $\widetilde{f}=f\circ \pi$. 
Let $x\in \Gamma$ and choose $\widetilde{x}\in \pi^{-1} (x)\subset F_m$ so that $l(\widetilde{x})=l(x)$. 
Since the ball of radius less than $g_j/2$ with center $\widetilde{x}$ in $X$ is isometrically
isomorphic to that of the same radius with center $x$ in $\mathrm{Cay}(\Gamma, S)$, 
the above inequality (for $\widetilde{f}$, $\widetilde{x}$) implies 
$$
E^{(l)}(f)(x)\leq \frac{8\, C^{12}\, D_j^{4\sigma}\, l(x)^{8\sigma}}{C'\lambda} E(f)(x), 
$$
that is, \eqref{nstepenergy-upperbound2}. 
\end{proof}

Theorem \ref{main_fixed_point_theorem2} now follows by combining Theorem \ref{existence}, 
Lemma \ref{growth} and Lemma \ref{nongrowth}. 

\section*{Appendix}

Let $\Gamma$ be a finitely generated group equipped with a finite, symmetric generating set $S$, 
and let $\rho\colon \Gamma\curvearrowright \mathcal{H}$ be an affine action, 
where $\mathcal{H}$ is a Hilbert space. 
In \S 3, we referred to the following fact: if $\rho(\Gamma)$ has no fixed point, then energy minimization 
coupled with scaling ultralimit argument produces a nonconstant map from $\Gamma$ to a (possibly new) 
Hilbert space $\mathcal{H}'$ which is equivariant with respect to a (possibly new) affine action 
$\rho'\colon \Gamma\curvearrowright \mathcal{H}'$ and minimizes the local energy at a point. 
While this fact would not be useful for our purpose of proving Theorem \ref{existence} 
as we observed that we would not be able to conclude the resulting map is harmonic, 
it might be so in other circumstances. 
Therefore, we shall verify the above fact by proving the following

\begin{Proposition}\label{minimizer}
Let $\Gamma$ be a finitely generated group equipped with a finite, symmetric generating set $S$, 
and let $\rho\colon \Gamma\curvearrowright \mathcal{H}$ be an affine action, 
where $\mathcal{H}$ is a Hilbert space. 
Suppose that $\rho(\Gamma)$ has no fixed point. 
Fix $x\in \Gamma$. 
Then there exist a (possibly new) affine action $\rho'\colon \Gamma\curvearrowright \hilbert'$, 
where $\hilbert'$ is a (possibly new) Hilbert space, 
and a nonconstant $\rho'$-equivariant map $f\colon \Gamma\rightarrow \hilbert'$ 
minimizing the local energy at $x$. 
If $\rho$ satisfies \eqref{eq:growth-new'} for some $C>0$ and $\sigma\geq 0$, then $\rho'$ 
also satisfies \eqref{eq:growth-new'} for the same $C$, $\sigma$. 
\end{Proposition}

Before proceeding to the proof, we review the definitions of ultrafilter and the ultralimit 
of a sequence of metric spaces. 

A nonempty subset $\omega \subset 2^\N$ is called an {\em ultrafilter} 
on $\N$ if it satisfies the following conditions:
\begin{enumerate}
\renewcommand{\theenumi}{\roman{enumi}}
\renewcommand{\labelenumi}{(\theenumi)}
\item $\emptyset \notin \omega$.
\item $A\in \omega,\,\, A\subset B \quad \Rightarrow\quad
B\in \omega$.
\item $A,B\in \omega\quad \Rightarrow\quad A\cap B \in \omega$.
\item For any subset $A\subset \N$, $A\in \omega$ or
$\N\setminus A\in \omega$.
\end{enumerate}
An ultrafilter $\omega$ on $\N$ is called {\em non-principal} 
if it satisfies also

\smallskip\noindent
\quad\, (v) For any finite subset $F\subset \N$, $F\notin \omega$
(hence, $\N\setminus F\in\omega$).

\smallskip
Let $\omega$ be a non-principal ultrafilter on $\N$. 
Let $(a_j)_{j=1}^\infty\subset \R$ be a sequence of real numbers. 
We call $\alpha\in \R$ an {\em $\omega$-limit} of $(a_j)$ and write 
$\omega\mbox{-}\lim_j a_j = \alpha$ if 
$\{ j\in \N \mid |a_j-\alpha|<\varepsilon \}\in \omega$ holds for any $\varepsilon >0$. 
Let $(Y_j, d_j, o_j)$ be a sequence of metric spaces with base point. 
On the set of sequences $(y_j)$, 
where $y_j\in Y_j$ and $d_j(o_j, y_j)$ is bounded independent of $j$, 
consider the equivalence relation $[(y_j) \sim (z_j)\,\,\Leftrightarrow
\,\, \omega\mbox{-}\lim_j d_j(y_j, z_j) = 0]$,
and denote the equivalence class of $(y_j)$ by 
$y_\infty = \omega\mbox{-}\lim_j y_j$. 
Let $Y_\infty$ denote the set of equivalence classes, and endow it with 
the metric $d_\infty (y_\infty, z_\infty)
= \omega\mbox{-}\lim_j d_j(y_j, z_j)$.
One writes $(Y_\infty, d_\infty, o_\infty) = \omega\mbox{-}\lim_j (Y_j, d_j, o_j)$, 
called the {\em $\omega$-limit} of $(Y_j, d_j, o_j)$. 
It is known that the metric space $(Y_\infty, d_\infty)$ is necessarily complete. 

\bigskip\noindent
{\em Proof of Proposition \ref{minimizer}}\,\, 
We shall follow \cite{Silberman2} and \cite{Kondo} which treat the case that 
the action is isometric. 

Fix a non-principal ultrafilter $\omega$ on $\N$. 
We divide the proof into two cases, according to whether $E_0 := \inf E(f)(x)$ 
is strictly positive or not, where the infimum is taken over all $\rho$-equivariant maps 
$f\colon \Gamma\rightarrow \hilbert$. 

\smallskip\noindent
{\bf Case 1.}\quad The case that $E_0>0$. 
  
This is a simpler case, and we only outline the argument. 
Let $\{ f_j \}_{j=1}^\infty$ be a sequence of $\rho$-equivariant maps 
$\Gamma\rightarrow \hilbert$ such that $E(f_j)(x) \searrow E_0$. 
Set $\mathbf{v}_j = f_j(x)$ and define $(\hilbert_\infty, \|\cdot\|_\infty, \mathbf{v}_\infty) 
= \omega\mbox{-}\lim_j (\hilbert, \|\cdot\|, \mathbf{v}_j)$. 
Then an affine action $\rho_\infty\colon \Gamma\curvearrowright \hilbert_\infty$ 
is induced and satisfies \eqref{eq:growth-new'}. 
Define a map $f_\infty\colon \Gamma\rightarrow \hilbert_\infty$ by 
$f_\infty(y) = \omega\mbox{-}\lim_j f_j(y)$ for $y\in \Gamma$. 
Then $f_\infty$ is $\rho_\infty$-equivariant, and 
$$
E(f_\infty)(x) = \omega\mbox{-}\lim_j E(f_j)(x) = E_0; 
$$
in particular, $f_\infty$ is nonconstant. 
On the other hand, one can verify that $E(g)(x)\geq E_0$ for all $\rho_\infty$-equivariant 
maps $g\colon \Gamma\rightarrow \hilbert_\infty$. 
Thus, $f_\infty$ minimizes the local energy at $x$. 

%Case1の証明を厳密にチェックする。

\smallskip\noindent
{\bf Case 2.}\quad The case that $E_0=0$.

Define $\delta\colon \hilbert\rightarrow \R_{\geq 0}$ by 
$\delta(\mathbf{v}) = \max_{s\in S} \| \rho(s)(\mathbf{v}) - \mathbf{v} \|$. 
While $\delta>0$ since $\rho(\Gamma)$ has no fixed-point, we have 
$\inf_{\mathbf{v}\in \hilbert} \delta(\mathbf{v}) = 0$; indeed, 
\begin{eqnarray*}
E(f)(x) &=& \frac{1}{2\# S} \sum_{s\in S} \| f(xs)-f(x) \|^2\\ 
&=& \frac{1}{2\# S} \sum_{s\in S} \| \rho(x) \{ \rho(sx^{-1}) (f(x)) 
- \rho(x^{-1})(f(x)) \} \|^2, 
\end{eqnarray*}
which is clearly comparable to $\delta(\rho(x^{-1})(f(x)))^2$. 

In order to proceed, we need the following elementary fact: let $Y$ be a complete metric space 
and $\varphi\colon Y\rightarrow \R$ a strictly positive continuous function.
Then there exists $y\in Y$ such that 
$d_Y(z,y)\leq \varphi(y)\,\, \Rightarrow\,\, \varphi(z)\geq \frac{1}{2}\varphi(y)$. 
Let $j\in \N$ and apply this fact to the function $j\delta\colon \hilbert\rightarrow \R$. 
Then we get $\mathbf{v}_j\in \hilbert$ such that 
$\| \mathbf{w} - \mathbf{v}_j \|\leq j\delta(\mathbf{v}_j)\,\, \Rightarrow\,\, 
\delta(\mathbf{w})\geq \frac{1}{2}\delta(\mathbf{v}_j)$. 
Now let 
$(\hilbert_\infty, \|\cdot\|_\infty, \mathbf{v}_\infty) = \omega\mbox{-}\lim_j 
\left(\hilbert, \frac{1}{\delta(\mathbf{v}_j)}\|\cdot\|, \mathbf{v}_j \right)$. 

We shall define an affine action $\rho_\infty\colon \Gamma\curvearrowright 
\mathcal{H}_\infty$. 
Let $\mathbf{w}_\infty\in \mathcal{H}_{\infty}$ and write $\mathbf{w}_\infty 
= \omega\mbox{-}\lim_j \mathbf{w}_j$. 
By definition, there exists $M>0$ such that 
$\|\mathbf{w}_j - \mathbf{v}_j\|\leq M \delta(\mathbf{v}_j)$ for all $j\in \N$. 
Then 
\begin{eqnarray*}
\| \rho(s)(\mathbf{w}_j) - \mathbf{v}_j \|
&\leq& \| \rho(s)(\mathbf{w}_j) - \rho(s)(\mathbf{v}_j) \| 
+ \| \rho(s)(\mathbf{v}_j) - (\mathbf{v}_j) \|\\ 
&\leq& C\, \| \mathbf{w}_j - \mathbf{v}_j \| + \delta(\mathbf{v}_j)\\ 
&\leq& (CM+1)\, \delta(\mathbf{v}_j), 
\end{eqnarray*}
where $C= \| A(s) \|$. 
It follows that $\omega\text{-}\lim_j \rho(s)(\mathbf{w}_j)$ exists, 
and it is easy to verify that this limit is independent of the choice of $\mathbf{w}_j$. 
Hence, by defining $\rho_\infty(s)(\mathbf{w}_\infty) = \omega\text{-}\lim_j \rho(s)(\mathbf{w}_j)$, 
we obtain a well-defined map 
$\rho_\infty(s)\colon \mathcal{H}_{\infty}\rightarrow \mathcal{H}_{\infty}$, 
which is clearly $C$-Lipschitz. 
It is also easy to see that the affineness, that is, the property of preserving internally dividing 
points, of $\rho(s)$ is inherited by $\rho_\infty(s)$. 
%Since each $\rho(s)$ preserves affine lines in $\mathcal{H}$ and affine lines 
%in $\mathcal{H}_{\infty}$ are ultralimits of those in $\mathcal{H}$, 
%$\rho_\infty(s)$ also preserves affine lines in $\mathcal{H}_{\infty}$. 
%Therefore, $\rho_\infty(s)$ is an affine transformation. 
Let $\gamma\in \Gamma$ and write $\gamma=s_1\dots s_l$, where $s_1,\dots, s_l \in S$. 
Let $\mathbf{w}_{\infty} = \omega\text{-}\lim_j \mathbf{w}_j\in \mathcal{H}_{\infty}$. 
Then the ultralimit of $\rho(\gamma)(\mathbf{w}_j) = \rho(s_1) \dots \rho(s_l)(\mathbf{w}_j)$ 
exists and equals to $\rho_\infty(s_1) \dots \rho_\infty(s_l) (\mathbf{w}_{\infty})$. 
Thus, defining 
$\rho_\infty(\gamma) (\mathbf{w}_{\infty}) = \omega\text{-}\lim_j \rho(\gamma)(\mathbf{w}_j)$, 
we have $\rho_\infty(\gamma) = \rho_\infty(s_1)\dots \rho_\infty(s_l)$ and obtain 
an affine action $\rho_\infty\colon \Gamma\curvearrowright 
\mathcal{H}_\infty$. 
%$\rho_\infty$ of $\Gamma$ on $\mathcal{H}_{\infty}$. 
It is clear that if $\rho$ satisfies \eqref{eq:growth-new'}, then $\rho_\infty$ also satisfies 
\eqref{eq:growth-new'} with the same constants. 

We now verify that $\delta_\infty\geq \frac{1}{2}$, where $\delta_\infty$ is the 
function $\delta$ with respect to $\rho_\infty$. 
To do so, take any $\mathbf{w}_\infty = \omega\mbox{-}\lim_j \mathbf{w}_j\in \hilbert_\infty$, 
so that $\| \mathbf{w}_j - \mathbf{v}_j \|\leq M\, \delta(\mathbf{v}_j)$ for some $M>0$, 
and set $A_s := \{ j\in \N \mid \| \rho(s)(\mathbf{w}_j) - \mathbf{w}_j \|\geq 
\frac{1}{2} \delta(\mathbf{v}_j) \}$ for $s\in S$. 
For $j> M$, $\| \mathbf{w}_j - \mathbf{v}_j \|\leq j\, \delta(\mathbf{v}_j)$, and therefore 
$\delta(\mathbf{w}_j)\geq \frac{1}{2} \delta(\mathbf{v}_j)$, that is, $j\in \cup_{s\in S} A_s$. 
Thus %$\N\setminus \cup_{s\in S} A_s$ is finite, and thus 
$\cup_{s\in S} A_s\in \omega$. 
But this means $A_s\in \omega$ for some $s\in S$. 
Therefore, $\| \rho_\infty(s)(\mathbf{w}_\infty) - \mathbf{w}_\infty \|_\infty
\geq \frac{1}{2}$, and $\delta_\infty\geq \frac{1}{2}$. 
We thus recover the situation of Case 1.
\hfill{$\square$}


\begin{thebibliography}{99}

\bibitem{ArzhantsevaDelzant} G.~Arzhantseva and T.~Delzant, {\em Examples of random groups}, preprint. 

\bibitem{BFGM} U.~Bader, A.~Furman, T.~Gelander and N.~Monod, {\em Property (T) and rigidity 
for actions on Banach spaces}, Acta Math. {\bf 198} (2007), 57--105.

\bibitem{CoornaertKnieper} M.~Coornaert and G.~Knieper, 
{\em Growth of conjugacy classes in Gromov hyperbolic groups}, 
Geom.~Funct.~Anal. {\bf 12} (2002), 464--478. 

\bibitem{Delorme} P.~Delorme, {\em $1$-cohomologie des repr\'{e}sentations unitaires des
groupes de Lie semi-simples et r\'{e}solubles. Produits tensoriels continus et 
repr\'{e}sentations}, Bull.~Soc.~Math.~France {\bf 105} (1977), 281--336.

\bibitem{Ghys} E.~Ghys, {\em Groupes Al\'{e}atoires [d'apr\`{e}s Misha Gromov,...]}, 
S\'{e}minaire Bourbaki, 55\`{e}me ann\'{e}e, 2002--2003, ${\rm n}^\circ$916. 

\bibitem{Gromov1} M.~Gromov, Asymptotic invariants of infinite groups, in Geometric group theory, 
ed.~G.~Niblo, M.~Roller, Cambridge University Press, Cambridge, 1993.

\bibitem{Gromov2} M.~Gromov, 
 {\em Random walk in random groups}, 
 Geom.~Funct.~Anal. {\bf 13} (2003), 73--146. 

\bibitem{Guichardet} A.~Guichardet, {\em \'{E}tude de la $1$-cohomologie et de la topologie
du dual pour les groupes de Lie \`{a} radical ab\'{e}lien}, Math.~Ann. {\bf 228} (1977), 215--232. 

\bibitem{IzekiKondoNayatani1} H.~Izeki, T.~Kondo, and S.~Nayatani, 
 {\em Fixed-point property of random groups}, 
 Annals of Global Analysis and Geom. {\bf 35} (2009), 363--379 

\bibitem{IzekiKondoNayatani2} H.~Izeki, T.~Kondo, and S.~Nayatani, 
 {\em $N$-step energy of maps and fixed-point property of random
	 groups}, Groups, Geometry, and Dynamics {\bf 6} (2012), 701--736. 

\bibitem{IzekiKondoNayataniprep} H.~Izeki, T.~Kondo, and S.~Nayatani, 
in preparation. 

\bibitem{IzekiNayatani} H.~Izeki and S.~Nayatani, 
{\em Combinatorial harmonic maps and discrete-group actions 
on Hadamard spaces}, Geom. Dedicata {\bf 114} (2005), 147--188.

\bibitem{Kazhdan} D.~Kazhdan, {\em Connection of the dual space of a group with the structure 
of its closed subgroups}, Funct.~Anal.~Appl. {\bf 1} (1967), 63--65. 

\bibitem{Kondo} T.~Kondo, {\em Fixed point theorems via a scaling limit argument} (in Japanese), 
RIMS Kokyuroku {\bf 1720} (2010), 139--149. 

\bibitem{KotowskiKotowski} M.~Kotowski and M.~Kotowski, 
{\em Random groups and property (T): Zuk's theorem revisited}, J.~Lond.~Math.~Soc. {\bf 88} (2013), 396--416. 

\bibitem{Mimura} M.~Mimura, private communication, April 14, 2015. 

\bibitem{NaorSilberman} A.~Naor and L.~Silberman, {\em Poincar\'{e} inequalities, embeddings, and wild groups}, 
Compos.~Math. {\bf 147} (2011), 1546--1572. 

\bibitem{Nowak1} P.~W.~Nowak, 
{\em Poincar\'{e} inequalities and rigidity for actions on Banach spaces}, 
J.~Eur.~Math.~Soc. {\bf 17} (2015), no. 3, 689--709. 

\bibitem{Nowak2} P.~W.~Nowak, 
Group actions on Banach spaces. Handbook of group actions. Vol.~II, 121--149, Adv.~Lect.~Math. {\bf 32}, 
Int.~Press, Somerville, MA, 2015.

\bibitem{Ollivier} Y.~Ollivier, {\em Sharp phase transition theorems for hyperbolicity of random groups}, 
Geom.~Funct.~Anal. {\bf 14} (2004), 595--679.

\bibitem{Silberman1} L.~Silberman, 
{\em Addendum to ``Random walk on random groups'' by M.~Gromov}, Geom.~Funct.~Anal. {\bf 13} (2003), 147--177. 

\bibitem{Silberman2} L.~Silberman, note formerly available at the author's homepage. 

%\bibitem{Zuk1}
%A. \.Zuk, {\em La propri\'et\'e {\rm (T)} de Kazhdan pour les groupes 
%agissant sur les poly\'edres}, C.~R.~Acad.~Sci.~Paris {\bf 323} (1996), 
%453--458. 

\bibitem{Zuk2} Zuk, A.: {\em Property {\rm (T)} and Kazhdan constants 
for discrete groups}. Geom.~Funct.~Anal. {\bf 13} (2003), 643--670. 

\end{thebibliography}
\end{document}